\documentclass[oneside,12pt]{amsart}

\usepackage{amscd,amsmath,latexsym,bm,pdfpages,verbatim,amsthm,accents}
\usepackage[mathcal]{euscript}
\usepackage{graphicx}
\usepackage{amssymb,enumerate}          
\usepackage{epstopdf}
\usepackage[all]{xy}

\newtheorem{thm}{Theorem}[section]

\newtheorem{lem}[thm]{Lemma}
\newtheorem{prop}[thm]{Proposition}

\theoremstyle{definition}

\theoremstyle{definition}
\newtheorem{construction}[thm]{Construction}

\theoremstyle{definition}

\newtheorem{remark}[thm]{Remark}

\theoremstyle{remark}
\newtheorem*{rem}{Remark}

\begin{document}
\def\X#1#2{r(v^{#2}\ds{\prod_{i \in #1}}{x_{i}})}
\def\skp#1{\vskip#1cm\relax}
\def\block{\rule{2.4mm}{2.4mm}}

\def\cred{\color{red}}
\def\cblue{\color{blue}}
\def\cblack{\color{black}}

\def\block{\rule{2.4mm}{2.4mm}}

\def\nd{\noindent}
\def\becomes{\colon\hspace{-2,5mm}=}
\def\ds{\displaystyle}
\def\red{\color{red}}
\def\black{\color{black}}
\def\s{\sigma}
\numberwithin{equation}{section}

\title[A generalization of the Davis-Januszkiewicz construction \ldots]{A generalization of the 
Davis-Januszkiewicz construction and applications to toric manifolds and iterated polyhedral products}

\skp{0.2}

\author[A.~Bahri]{A.~Bahri}
\address{Department of Mathematics,
Rider University, Lawrenceville, NJ 08648, U.S.A.}
\email{bahri@rider.edu}

\author[M.~Bendersky]{M.~Bendersky}
\address{Department of Mathematics
CUNY,  East 695 Park Avenue New York, NY 10065, U.S.A.}
\email{mbenders@hunter.cuny.edu}

\author[F.~R.~Cohen]{F.~R.~Cohen}
\address{Department of Mathematics,
University of Rochester, Rochester, NY 14625, U.S.A.}
\email{cohf@math.rochester.edu}

\author[S.~Gitler]{S.~Gitler}
\address{El Colegio Nacional, Gonzalez Obregon 24\,C, Centro Historico, Mexico City, Mexico.}  \email{sgitler@math.cinvestav.mx}



\

\begin{abstract}
The fundamental Davis-Januszkiewicz construction of toric manifolds is reinterpreted in order to allow for generalization. Applications involve the simplicial wedge $J$-construction and Ayzenberg's recent
identities arising from composed simplicial complexes. \end{abstract}

\maketitle
\tableofcontents
\section{Introduction}\label{sec:introduction}
The topological approach to non-singular toric varieties requires two ingredients:
\begin{enumerate}
\item a simple polytope $P^{n}$  of dimension $n$ having 
a set $\mathcal{F}$ of $m$ facets and
\item a {\em characteristic\/} function $\lambda \colon \mathcal{F} \longrightarrow \mathbb{Z}^n$
which assigns an integer vector to each facet of the simple polytope $P^n$. 
\end{enumerate}

\nd The latter can be considered as an $(n \times m)$-matrix 
$\lambda \colon \mathbb{Z}^m \longrightarrow \mathbb{Z}^n$
\nd with integer entries and columns indexed by the facets of  $P^n$.
A regularity condition, which ensures the smoothness of the toric manifold, requires 
all $n \times n$ minors of $\lambda$ corresponding to the vertices of $P^n$ to be $+1$ or $-1$.

Associated to the pair $(P^{n},\lambda)$,  Davis and Januszkiewicz \cite{davis.jan}, constructed two spaces:
$$\mathcal{L} = T^m \times P^n \big/\!\!\sim$$

\nd and a toric manifold
$$M^{2n} = T^n \times P^n \big/\!\!\sim_{\!\lambda}.$$

\nd The properties of the spaces $\mathcal{L}$ have been studied extensively via
an alternative general construction developed by Buchstaber and Panov \cite{buchstaber.panov.2}, 
who gave them the name  ``moment-angle complexes''
$$\mathcal{L} \;=\; T^{m}\times P^{n}\big/\!\!\sim \quad\cong\quad Z(K_{P}; (D^{2},S^{1})).$$

\nd In the notation used here, $K_{P}$ represents the simplicial complex dual to the boundary of a simple
polytope $P^{n}$. From this point of view, the toric manifold $M^{2n}$ is recovered as the quotient
$Z(K_{P}; (D^{2},S^{1}))\big/\rm{ker}\hspace{0.02in}\lambda$.

The main results presented in authors' earlier work \cite{bbcg3}, arise from a construction on a simplicial 
complex $K_{P}$ having $m$ vertices. For each sequence
$J = (j_1, j_2,\ldots,j_m)$ of positive integers, a new simplicial complex $K_{P}(J)$ is constructed,
$$K_{P} \rightsquigarrow K_{P}(J).$$

\nd Also, associated to $P^{n}$ is another simple polytope $P(J)$ and $K_{P}(J) \;=\; K_{P(J)}$.
Everything fits together in such a way that, from the toric manifold $(P^{n}, \lambda, M^{2n})$, it is possible to construct  another toric manifold $\big(P(J),\lambda(J), M(J)\big)$.  In the context of moment-angle complexes
and polyhedral products (\cite{bbcg}), it is shown in \cite{bbcg3} that there is a diffeomorphism of orbit
spaces
\begin{equation*}\label{eqn:orbit.homeo}
Z(K_{P};(\underline{D}^{2J}, \underline{S}^{2J-1}))\big/\text{{\rm ker}}\;\!\lambda
\; \longrightarrow \; Z\big(K_{P}(J); (D^2, S^1))\big)\big/\text{{\rm ker}}\;\lambda(J)
\end{equation*}

\nd which defines $M(J)$. Here, the left hand side uses the notation of a polyhedral product,  
({\em generalized\/} moment-angle complex), associated to the family of pairs
$$(\underline{D}^{2J}, \underline{S}^{2J-1}) \;=\; \big\{(D^{2j_{i}}, S^{2j_{i}})\big\}_{i=1}^{m}.$$

\nd These ideas are elaborated upon in Section \ref{sec:families}. The right hand side involves
an ordinary moment-angle complex and fits in with the formalism of Davis--Januszkiewicz and
Buchstaber--Panov. A natural question arises about the left hand side which involves the combinatorics
of the smaller simplicial complex $K_{P}$: where is this visible in the standard Davis--Januszkiewicz
construction of toric manifolds? One of the primary goals here is to answer this question.

\

The details of the Davis--Januszkiewicz construction are reviewed in Section \ref{sec:djreview}
and the modifications necessary to generalize the construction are described in Section 
\ref{sec:modify.relations}. The modified construction is interpreted in terms of polyhedral products in
Section \ref{sec:bp.formalism} and the answer to the question posed above is presented in
Section  \ref{sec:families}. Additional applications involving the ``composed complex'' constructions
of A.~Ayzenberg \cite{aa}, are discussed in Sections \ref{sec:iterated.pp} and \ref{sec:further.pp}.

\

\nd{\bf Acknowledgments.} The first author was supported in part by a Rider University Summer Research Fellowship and grant number 210386 from the Simons Foundation; the third author was supported partially by 
DARPA grant number 2006-06918-01.

\

\section{A review of the Davis-Januszkiewicz construction}\label{sec:djreview}

A toric manifold $M^{2n}$ is a manifold covered by local 
charts $\mathbb{C}^n$, each with the standard  action of a real $n$-dimensional torus
$T^n$, compatible in such a way that 
the quotient $M^{2n}\big/T^n$ has the structure of a  {\em simple\/} polytope $P^{n}$. Here, ``simple'' 
means that  $P^n$  has the property that at each vertex, exactly $n$ facets  intersect.
Under the $T^n$ action, each copy of $\mathbb{C}^n$ must project to an $\mathbb{R}^n_+$  
neighborhood of a vertex of $P^n$. The fundamental  construction of Davis and Januszkiewicz 
\cite[Section $1.5$]{davis.jan} is described briefly below. It realizes all toric manifolds and, in 
particular, all smooth projective toric varieties. Let
$${\mathcal F} = \{F_{1},F_{2},\ldots,F_{m}\}$$

\nd denote the set of facets of $P^n$. The fact that $P^n$ $\;$ is 
simple implies that every codimension-$l$ face $F$ can be 
written uniquely as
$$F = F_{i_{1}} \cap F_{i_{2}} \cap \cdots \cap F_{i_{l}}$$

\nd where the $F_{i_{j}}$ are the facets containing $F$. Let
\begin{equation}\label{eqn:lambda}
\lambda : {\mathcal F} \longrightarrow \mathbb{Z}^n
\end{equation}

\nd be a function into an $n$-dimensional integer lattice satisfying the {\em regularity\/}
condition that whenever 
$F = F_{i_{1}} \cap F_{i_{2}} \cap \cdots \cap F_{i_{l}}$ then 
$\{\lambda(F_{i_{1}}),\lambda(F_{i_{2}}), \ldots ,\lambda(F_{i_{l}})\}$ span 
an $l$-dimensional submodule of $\mathbb{Z}^n$ which is a direct summand.
Such a map is called a {\em characteristic function\/} associated to $P^n$.
Next, regarding $\mathbb{R}^n$ as the Lie algebra of $T^n$,  the map
$\lambda$ is used to associate  to each codimension-$l$ face $F$ of $P^n$ \/ a rank-$l$
subgroup $G_F \subset T^n$. Specifically, writing
$$\lambda(F_{i_j}) = (\lambda_{1{i_j}},\lambda_{2{i_j}},\ldots,\lambda_{n{i_j}})$$

\nd gives
$$ G_F = \big\{\big(e^{2\pi{i}(\lambda_{1{i_1}}t_1 +\lambda_{1{i_2}}t_2 + \cdots + \lambda_{1{i_l}}t_l)},
\ldots, e^{2\pi{i}(\lambda_{n{i_1}}t_1 + \lambda_{n{i_2}}t_2 + \cdots +\lambda_{n{i_l}}t_l)}\big) \in T^n\big\}$$
\skp{0.19}
\nd where $t_i \in \mathbb{R},\, i = 1,2,\ldots,l$. Finally, let $p \in$ $P^n$ $\;$ and $F(p)$ be 
the unique face with $p$ in its relative interior. Define an equivalence
relation $\sim_{\!\lambda}$ on $T^n$ $\times$ $P^n$ $\;$ by $(g,p) \sim_{\!\lambda} (h,q)$ if and only
if $p = q$ and $g^{-1}h \in G_{F(p)} \cong T^l$. Then
\begin{equation}\label{eqn:defn.tm}
M^{2n} \cong M^{2n}(\lambda) = T^n \times P^n\big/\!\sim_{\!\lambda}
\end{equation} 

\nd is a smooth, closed, connected, $2n$-dimensional manifold with 
$T^n$ action induced by left translation \cite[page 423]{davis.jan}. A 
projection $\pi \colon M^{2n} \rightarrow P^n$ onto the polytope is induced from the projection
$T^n \times$ $P^n$ $\rightarrow$ $P^n$. 
\begin{rem}
In the cases when $M^{2n}$ is a projective non-singular toric
variety, $P^n$ \;and $\lambda$ encode topologically the information in the defining fan, 
\cite[Chapter 5]{buchstaber.panov.2}.
\end{rem}

Let $K_{P}$ denote the simplicial complex dual to the boundary of simple polytope 
$P^n$ having $m$ facets. Recall
that the duality here is in the sense that the facets of $P^n$ correspond to the vertices of 
$K_{P}$.  A set of vertices in $K_{P}$ is a simplex if and only if the corresponding facets in
$P^n$ all intersect.  
Davis and Januszkiewicz constructed a second space in \cite{davis.jan}, which came to be known as a 
{\em moment-angle manifold\/}, by
\begin{equation}\label{eqn:mac}
\mathcal{Z} \; \simeq\; T^m \times P^n \big/\!\!\sim
\end{equation}

\nd where here $\sim$ does not involve the characteristic $\lambda$ but the combinatorics
of the simplicial complex $K_{P}$ only. Here also, the circles in $T^{m}$ are indexed by the facets of $P^{n}$.
The equivalence relation $\sim$ is defined by analogy with that of \eqref{eqn:defn.tm}.
Specifically, $\lambda$ in \eqref{eqn:lambda} is replaced by 
\begin{equation}\label{eqn:theta}
\theta\colon {\mathcal F} \longrightarrow \mathbb{Z}^m\end{equation}

\nd where $\theta(F_{i}) = \underline{e}_{i} \in \mathbb{Z}^{m}$. 

\

Constructions \eqref{eqn:mac} and \eqref{eqn:defn.tm} are related by a quotient map given by the free action of
$\text{ker}\hspace{0.02in}\lambda$ on $\mathcal{Z}$
\begin{equation}\label{eqn:bp.proj}
T^m \times P^n \big/\!\!\sim \;\longrightarrow\; (T^m \times P^n \big/\!\!\sim)\big/\text{ker}\hspace{0.02in}\lambda
\;\cong \;T^n \times P^n\big/\!\sim_{\!\lambda}
\end{equation}

\nd as described in \cite[Section $6.1$]{buchstaber.panov.2}.

\section{Modifying the equivalence relations}\label{sec:modify.relations}
As above, let $P^{n}$ be simple polytope. The construction of \eqref{eqn:mac} is generalized easily
by first replacing each of the circles in $T^{m}$ by spaces $X_1, X_2, \ldots,X_m$, indexed by the facets of $P^{n}$.

\

\begin{construction}\label{def:tilde1}
Define an equivalence relation $\sim_{1}$ on the Cartesian product 
$$X_1 \times X_2 \times \cdots \times X_m \times P^{n}\;$$

\nd as follows:
$$(x_{1},x_{2},\ldots,x_{m},p)\; \sim_{1}\; (y_{1},y_{2},\ldots,y_{m},q)$$

\nd if and only if: 
\begin{enumerate}[(a)]\itemsep3pt
\item $p =q$ and
\item when $p$ is in the relative interior of the face $F(p) = F_{j_{1}}\cap F_{j_{2}}\cap \cdots F_{j_{k}}$
given as the intersection of the $k$ facets which are complementary to 
$\{F_{i_{1}}, F_{i_{2}},\ldots, F_{i_{m-k}}\}$, then $x_{i_{s}} =  y_{i_{s}}$ for all $s \in \{1,2,\ldots,m-k\}$.
\end{enumerate}

\nd Equivalence classes of points in\; $(X_{1} \times X_{2} \times \cdots \times X_{m}) \times P^{n}\big/\!\!\sim_{1}$
\;are denoted by the symbol $\big[(x_{1},x_{2},\ldots,x_{m},p)\big]_{1}$.
\end{construction}


\

Suppose now that  $S^1$ acts {\em freely\/} on the spaces  $X_1, X_2, \ldots,X_m$, giving 
an action of $T^m$ on
$X_1 \times X_2 \times \cdots \times X_m$ in the obvious way. Recall  that the function $\theta$ of
\eqref{eqn:theta} indexes  the ``coordinate'' circles
in $T^{m}$  by the facets of $P^{n}$. Also, each space
$X_{i}$ is associated with the facet $F_{i}$, So, an intersection of $k$ facets
in $P^n$ determines a projection $T^{m} \longrightarrow T^{m-k}$ and, by this projection,   
$T^m$ acts on the product $X_{i_{1}} \times X_{i_{2}} \times \cdots \times X_{i_{m-k}}$.

\

Next, let $\lambda$ be a characteristic map  specified for the polytope $P^{n}$. Then
$$\text{ker}\hspace{0.02in}\lambda \;\cong\; T^{m-n}\; \subset\; T^m.$$

\nd For $k \leq n$, there is the induced action of $\text{ker}\hspace{0.02in}\lambda \subset\; T^m$  on the product
$$X_{i_{1}} \times X_{i_{2}} \times \cdots \times X_{i_{m-k}}$$

\nd  and a projection
\begin{equation}\label{eqn:quotients}
\pi_{i_{1},i_{2},\ldots, i_{m-k}}\colon 
 \; X_1 \times X_2 \times \cdots \times X_m\big/\rm{ker}\hspace{0.02in}\lambda
 \; \longrightarrow 
X_{i_{1}} \times X_{i_{2}} \times \cdots \times X_{i_{m-k}}\big/\rm{ker}\hspace{0.02in}\lambda
\end{equation}

\nd corresponding to each intersection of $k$ facets. 

\

\nd Equivalence classes of points in 
$X_1 \times X_2 \times \cdots \times X_m\big/\rm{ker}\hspace{0.02in}\lambda$ are denoted by the symbol
$[x_{1},x_{2},\ldots,x_{m}]_{\lambda}$. The next Construction generalizes that of
\eqref{eqn:defn.tm}.

\begin{construction}\label{def:tilde2}
Define an equivalence relation $\sim_{2}$ on the Cartesian product 
$$\big(X_1 \times X_2 \times \cdots \times X_m\big/\rm{ker}\hspace{0.02in}\lambda\big) \times P^n$$

\nd as follows:
$$\big([x_{1},x_{2},\ldots,x_{m}]_{\lambda},p\big) \sim_{2} \big([y_{1},y_{2},\ldots,y_{m}]_{\lambda},q\big)$$

\nd if and only if: 
\begin{enumerate}[(i)]\itemsep3pt
\item $p =q$
\item when $p$ is in the relative interior of the face $F(p) = F_{j_{1}}\cap F_{j_{2}}\cap \cdots F_{j_{k}}$
given as the intersection of the $k$ facets which are complementary to 
$\{F_{i_{1}}, F_{i_{2}},\ldots, F_{i_{m-k}}\}$, then
$$\pi_{i_{1},i_{2},\ldots, i_{m-k}}([x_{1},x_{2},\ldots,x_{m}]_{\lambda}) = 
\pi_{i_{1},i_{2},\ldots, i_{m-k}}([y_{1},y_{2},\ldots,y_{m}]_{\lambda}).$$
\end{enumerate}
\end{construction}

\nd Equivalence classes of points in\; 
$\big(X_1 \times X_2 \times \cdots \times X_m\big/{\rm{ker}}\hspace{0.02in}\lambda\big)\times P^n \big/\!\!\sim_{2}$
\;are denoted by the symbol $\big[([x_{1},x_{2},\ldots,x_{m}]_{\lambda},p)\big]_{2}$.

\

Construction \ref{def:tilde2} can be reinterpreted as follows.
The group $\text{ker}\hspace{0.02in}\lambda$ acts on the space $(X_{1} \times X_{2} \times \cdots \times X_{m}) 
\times P^{n}\big/\!\!\sim_{1}$\; by
$$t\cdot [(x_{1},x_{2}, \ldots,x_{m}, p)]_{1} \;=\; [t\cdot (x_{1},x_{2}, \ldots,x_{m}), p)]_{1}.$$

\nd Property (b) in Construction \ref{def:tilde1} ensures that the action is well defined.  The next lemma, the 
analogue of \eqref{eqn:bp.proj}, follows naturally.

\

\begin{lem}\label{lem:h}
There is a homeomorphism
$$h\colon\big(X_1 \times X_2 \times \cdots \times X_m\big/{\rm{ker}}\hspace{0.02in}\lambda\big) 
\times P^n \big/\!\!\sim_{2} \quad\longrightarrow\quad
\big((X_{1} \times X_{2} \times \cdots \times X_{m}) \times P^{n}\big/\!\!\sim_{1}\big)
\big/{\rm{ker}}\hspace{0.02in}\lambda$$

\nd given by 
$$h\big(\big[([x_{1},x_{2},\ldots,x_{m}]_{\lambda},p)\big]_{2}\big) 
\;=\; \big[[(x_{1},x_{2},\ldots,x_{m},p)]_{1}\big]_{\lambda}.$$
\end{lem}
\skp{0.2}
\begin{proof}
To see that $h$ is well defined, suppose
$$\big[([x_{1},x_{2},\ldots,x_{m}]_{\lambda},p)\big]_{2} \;=\; \big[([y_{1},y_{2},\ldots,y_{m}]_{\lambda},p)\big]_{2}$$

\nd with $p$ is in the relative interior of the face $F(p) = F_{j_{1}}\cap F_{j_{2}}\cap \cdots F_{j_{k}}$
given as the intersection of the $k$ facets which are complementary to 
$\{F_{i_{1}}, F_{i_{2}},\ldots, F_{i_{m-k}}\}$,
as in Construction \ref{def:tilde2}. Then, $\pi_{i_{1},i_{2},\ldots, i_{m-k}}([x_{1},x_{2},\ldots,x_{m}]) = 
\pi_{i_{1},i_{2},\ldots, i_{m-k}}([y_{1},y_{2},\ldots,y_{m}])$ and hence,
$$t\cdot (x_{i_{1}},x_{i_{2}},\ldots,x_{i_{m-k}}) \;=\; (y_{i_{1}},y_{i_{2}},\ldots,y_{i_{m-k}})$$

\nd for some $t \in \rm{ker}\hspace{0.02in}\lambda$. It follows now from Construction \ref{def:tilde1} that
$$\big[([x_{1},x_{2},\ldots,x_{m}]_{\lambda},p)\big]_{2}\;=\; \big[([y_{1},y_{2},\ldots,y_{m}]_{\lambda},p)\big]_{2}$$

\nd as required. 
\skp{0.2}
To check that $h$ is an injection, suppose that
$$h\big(\big[([x_{1},x_{2},\ldots,x_{m}]_{\lambda},p)\big]_{2}\big) 
\;=\; h\big(\big[([y_{1},y_{2},\ldots,y_{m}]_{\lambda},p)\big]_{2}\big).$$

\nd Then $t \in \rm{ker}\hspace{0.02in}\lambda$ exists so that
$$t\cdot [(x_{1},x_{2}, \ldots,x_{m}, p)]_{1}  \;=\; [t\cdot (x_{1},x_{2}, \ldots,x_{m}), p)]_{1} \;=\;  [(y_{1},y_{2}, \ldots,y_{m}, p)]_{1}.$$ 

\nd Next, write $t\cdot (x_{1},x_{2}, \ldots,x_{m}) = (u_{1},u_{2}, \ldots,u_{m})$. 
Since  $p$ is in the relative interior of the face $F(p) = F_{j_{1}}\cap F_{j_{2}}\cap \cdots F_{j_{k}}$, it follows that
$$u_{i_{s}} =  y_{i_{s}}\quad\text{for all}\quad s \in \{1,2,\ldots,m-k\}$$

\nd corresponding to the complementary facets. It follows that
$$\pi_{i_{1},i_{2},\ldots, i_{m-k}}([x_{1},x_{2},\ldots,x_{m}]_{\lambda}) = 
\pi_{i_{1},i_{2},\ldots, i_{m-k}}([y_{1},y_{2},\ldots,y_{m}]_{\lambda})$$

\nd and hence 
$$\big[([x_{1},x_{2},\ldots,x_{m}]_{\lambda},p)\big]_{2} \;=\; \big[([y_{1},y_{2},\ldots,y_{m}]_{\lambda},p)\big]_{2}.$$

\nd It is easy now to see that $h$ is a homeomorphism.\end{proof}

\section{Interpreting the generalizations in the Buchstaber--Panov formalism}\label{sec:bp.formalism}
Construction \eqref{eqn:mac} can be analyzed locally. In the neighbourhood of a vertex $v_{i}$, a simple polytope 
$P^{n}$ looks like $\mathbb{R}^{n}_{+}$.

\nd \hspace{1.7in}\includegraphics[width=2in]{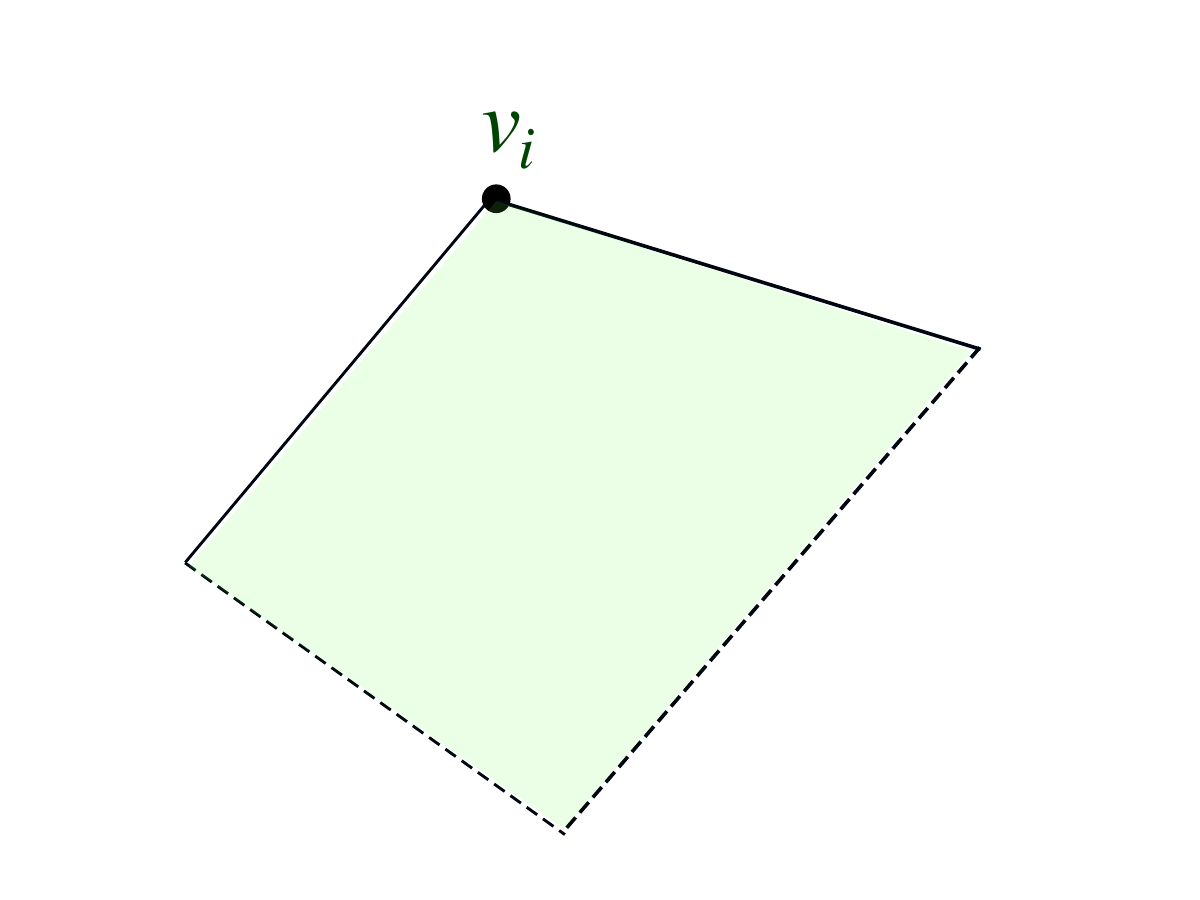}
\skp{-1.5}\hspace{3.2in}{\large $\mathbb{R}^{n}_{+}$}

\skp{1.3}

\nd The polytope can be given a {\em cubical\/} structure as in 
\cite[Construction $5.8$ and Lemma $6.6$]{buchstaber.panov.2}. The cube $I^{n}$, {\em anchored\/} by the vertex $v_{i}$, sits
inside the copy of $\mathbb{R}^{n}_{+}$ obtained by deleting all faces of $P^{n}$ which do not contain $v_{i}$.

\

Locally, $T^{m}\times P^{n}$ is
\begin{equation}\label{eqn:local}
T^{m}\times I^{n} \;\cong\; (S^{1} \times I)^{n} \times (S^{1})^{m-n}
\end{equation}

\nd Recall that all the circles in $T^{m}$ are indexed by the facets of the polytope so here, {\em the order
of the factors is confused\/}. The factors $S^{1}$ which are paired  with a copy of $I$ are those 
corresponding to the facets of $P^{n}$ which meet at $v_{i}$. The effect of the equivalence relation
$\;\sim\;$ in $T^{m}\times I^{n}\big/\!\!\sim$,\; \eqref{eqn:mac},
is to convert every $S^{1} \times I$ on the right hand side of \eqref{eqn:local} into a 
disc by collapsing $S^{1}\times \{0\}$ to a point. So
\begin{equation}\label{eqn:blocks}
T^{m}\times I\big/\!\!\sim\quad\cong\quad (D^{2})^{n} \times (S^{1})^{m-n}.
\end{equation}

The vertices of $P^{n}$ correspond  the maximal simplices of the simplicial complex $K_{P}$ so, assembling the blocks
\eqref{eqn:blocks} gives the moment-angle manifold
$$\mathcal{Z} \;=\; T^{m}\times P^{n}\big/\!\!\sim \quad\cong\quad Z(K_{P}; (D^{2},S^{1})).$$

\nd As described in \cite{davis.jan} and \cite{buchstaber.panov.2}, the map $\lambda$ determines a subtorus  
$\text{ker}\hspace{0.02in}\lambda =  T^{m-n} \subset T^{m}$ and a 
commutative diagram of quotient maps

\begin{equation}\label{eqn:djdiagram}
\begin{CD}
T^{m}\times P^{n}\big/\!\!\sim  @>{}>> T^n \times P^n \big/\!\!\sim_{\!\lambda}\\
@VV{\cong}V           @VV{\cong}V \\
 Z(K_{P}; (D^{2},S^{1})) @>{}>> Z(K_{P}; (D^{2},S^{1}))\big/\text{ker}\hspace{0.02in}\lambda. 
\end{CD}
\end{equation}

\

The construction above is generalized easily by replacing each of the circles in $T^{m}$ by
spaces $X_1, X_2, \ldots,X_m$ indexed by the facets of $P^{n}$.
Again, locally in the neighbourhood of a vertex $v_{i}$,  at which facets 
$F_{i_{1}}, F_{i_{2}},\ldots,F_{i_{n}}$ meet,
$X_1 \times X_2 \times \cdots \times X_m \times P^{n}\;$ is
\begin{align*}
X_1 \times X_2 \times &\cdots \times X_m \times I^{n}\\
&= (X_{i_{1}} \times X_{i_{2}} \times \cdots \times X_{i_{n}} \times I^{n})
\;\times\;\; X_{i_{n+1}} \times X_{i_{n+2}} \times \cdots \times X_{i_{m}}\\
&= (X_{i_{1}} \times I)\times (X_{i_{2}} \times I)\times 
\cdots \times (X_{i_{n}} \times I)
\;\times\;  X_{i_{n+1}}  \times X_{i_{n+2}} \times \cdots \times X_{i_{m}}
\end{align*}

\

\nd  As everything in the Cartesian product is indexed by the facets of the polytope, 
the order of the factors here has been shuffled naturally. Finally,  the equivalence relation  $\sim_{1}$ on
$$(X_{1} \times X_{2} \times \cdots \times X_{m}) \times I^{n}$$

\nd converts every $X_{i_{k}}\times I\;$ into $CX_{i_{k}}$.
 in a natural way. So,
\begin{align*}
(X_{1} \times X_{2} \times &\cdots \times X_{m}) 
\times I^{n}\big/\!\!\sim_{1}\\
&\cong\quad   CX_{i_{1}} \times CX_{i_{2}} \times 
\cdots \times CX_{i_{n}} \times (X_{i_{n+1}}  \times X_{i_{n+2}} \times \cdots \times X_{i_{m}} )
\end{align*}

\nd Assembling over all the vertices of $P^{n}$ along the common intersection determined by the cubical structure on $P^{n}$,
\nd gives the polyhedral product\:
$$ Z\Big(K_{P}; \big(\underline{CX_{i}}, \;
\underline{X_{i}}\big)\Big) \quad \subseteq \quad CX_{1} \times CX_{2} \times \cdots \times CX_{m}.$$

\nd just as in the case $X_{i} = S^{1}$ for standard moment-angle complexes, \cite[Section $6.1$]{buchstaber.panov.2}.

\

A choice of cubical structure for the simple polytope $P^{n}$ allows a choice of  homeomorphism
\begin{equation}\label{eqn:alpha}
\alpha\colon X_1 \times X_2 \times \cdots \times X_m
\times P^n \big/\!\!\sim_{1} \quad\longrightarrow\quad
Z\big(K_{P}; \big(\underline{CX}, \;\underline{X}\big)\big).
\end{equation}

\

The fact that $S^1$ acts (freely) on $X_i$ is now used again. $X_{i}$. 
The action extends to $CX_{i}$ by preserving the cone parameter. 
Consequently, there is an action of $T^{m}$ on
$CX_{1}\times CX_{2} \times \cdots\times CX_{m}$ which extends to
$$Z\Big(K_{P}; \big(\underline{CX}, \;\underline{X}\big)\Big) \;\subseteq\; 
 CX_{1} \times CX_{2} \times \cdots \times CX_{m}.$$

\nd In exactly the same way as it does for the case of $\text{ker}\hspace{0.02in}\lambda$ acting on 
$Z(K_{P}; (D^{2},S^{1}))$, the regularity condition on the characteristic map $\lambda$, (Section \ref{sec:djreview}), 
ensures that
$\text{ker}\hspace{0.02in}\lambda \subset T^{m}$ acts freely on 
$Z\Big(K_{P}; \big(\underline{CX}, \;\underline{X}\big)\Big)$. 

\begin{rem}
With respect to this action of $\text{ker}\hspace{0.02in}\lambda$, the homeomorphism $\alpha$ above, 
is equivariant and induces a homeomorphism of orbit spaces
$$\overline{\alpha}\colon \big((X_{1} \times X_{2} \times \cdots \times X_{m}) \times P^{n}\big/\!\!\sim_{1}\big)
\big/{\rm{ker}}\hspace{0.02in}\lambda \;\longrightarrow\;
Z\big(K_{P}; (\underline{CX}, \;\underline{X})\big)\big/\rm{ker}\hspace{0.02in}\lambda.$$
\end{rem}

The next theorem is the analogue of \eqref{eqn:djdiagram}.

\begin{thm}\label{thm:diagram}
The diagram following commutes

\begin{equation}\label{eqn:new.diagram}
\begin{CD}
X_1 \times X_2 \times \cdots \times X_m \times P^{n}\big/\!\!\sim_{1}  @>{\beta}>> 
\big(X_1 \times X_2 \times \cdots \times X_m\big/{\rm{ker}}\hspace{0.02in}\lambda\big) 
\times P^n \big/\!\!\sim_{2}\\
@V\alpha V{\cong}V           @V\gamma V{\cong}V \\
Z\big(K_{P}; (\underline{CX}, \;\underline{X})\big) @>{\delta}>> 
Z\big(K_{P}; (\underline{CX}, \;\underline{X})\big)\big/\rm{ker}\hspace{0.02in}\lambda 
\end{CD}
\end{equation}

\
  
\nd where the map $\beta$ is the composite of the quotient map
$$\big((X_{1} \times X_{2} \times \cdots \times X_{m}) \times P^{n}\big/\!\!\sim_{1}\big)
\;\longrightarrow 
\big((X_{1} \times X_{2} \times \cdots \times X_{m}) \times P^{n}\big/\!\!\sim_{1}\big)
\big/{\rm{ker}}\hspace{0.02in}\lambda$$

\nd with the map $h^{-1}$ of Lemma \ref{lem:h}, and $\gamma$ is the composite $\overline{\alpha}\circ h^{-1}$.
\end{thm}

\begin{proof} 
The homeomorphism $h$ can be used to replace the space in the top right hand corner with the  space
$\big((X_{1} \times X_{2} \times \cdots \times X_{m}) \times P^{n}\big/\!\!\sim_{1}\big)
\big/{\rm{ker}}\hspace{0.02in}\lambda$. This gives a new commutative diagram by the equivariance of the homeomorphism
$\alpha$. The maps $\beta$ and $\delta$ are defined in terms of the map $h^{-1}$ and so the diagram commutes as given.
\end{proof}

The next remark confirms that the original Davis-Januszkiewicz constructions are preserved.

\

\begin{remark}
For the case $X_{i} = S^{1}$ and $S^{1}$ acting on itself in the usual way, Constructions \ref{def:tilde1} and \ref{def:tilde2}
agree with those of \eqref{eqn:mac} and \eqref{eqn:defn.tm}.
\end{remark}

\section{Application to the construction of infinite families of toric manifolds}\label{sec:families}
\subsection{The case of odd spheres}\label{subsec:oddspheres} 
The first application  is to the infinite families of toric manifolds constructed in \cite{bbcg3} and summarized briefly below.
\skp{0.3}
Let $K$ be a simplicial complex of dimension $n-1$ on vertices $\{v_1,v_2,\ldots,v_m\}$. Given a sequence of positive
integers $J = (j_1, j_2,\ldots,j_m)$, define a new simplicial complex $K(J)$  on new vertices 
$$\big\{{v_{11},\ldots,v_{1j_1}}, {v_{21},\ldots,v_{2j_2}},
\ldots,{v_{m1},\ldots,v_{mj_m}}\big\},$$

\nd with the property that 
$$\big\{{v_{i_{1}1},\ldots,v_{i_{1}j_{i_{1}}}},\;\;\ldots\;\;,{v_{i_{k}1},\ldots,v_{i_{k}j_{i_{k}}}}\big\}$$

\nd is a  minimal non-face of $K(J)$ if and only if  $\{v_{i_1},\ldots,v_{i_k}\}$ 
is a minimal non-face of $K$.  Moreover, all minimal non-faces of $K(J)$ have this form.  A result of 
Provan and Billera, \cite[page 578]{pb}, ensures
that If $K = K_{P}$ is dual to the boundary of a simple polytope $P$, then $K(J)$ is dual to the boundary of
another simple polytope $P(J)$. That is
\begin{equation}\label{eqn:kp}
K_{P}(J) \;=\; K_{P(J)}.
\end{equation}

\nd It is the case also that the polytope $P(J)$ can be constructed directly from P in a straightforward way.

\skp{0.3}

Let $(P^{n}, \lambda, M^{2n})$ specify a toric manifold as in \eqref{eqn:defn.tm}. 
\nd From this, it is possible to construct another toric manifold $\big(P(J),\lambda(J), M(J)\big)$  where the numbers
$m$ and $n$, from Section \ref{sec:djreview}, transform as follows.

\begin{equation}\label{eqn:transform1}
\left[\begin{array}{c}
m\\
n\\
m-n
\end{array}\right]
\rightsquigarrow
\left[\begin{array}{c}
d(J) = j_{1}+j_{2}+\cdots+j_{m}\\
d(J) -m+n\\
m-n
\end{array}\right]
\end{equation}

\

\nd In terms of the original characteristic map $\lambda$, the matrix specifying the characteristic map 
$\lambda(J)$ 
is given in Figure 1 below. In that figure, $I_{j_{i}-1}$ represents the the identity matrix of size $j_{i}-1$. 
It is clear from the form of the matrix $\lambda(J)$ and the definition of $K(J)$, that the 
Davis-Januszkiewicz cohomology calculation, \cite[Theorem $4.14$]{davis.jan}, expresses the integral cohomology of $M(J)$ in terms of the sequence $J$, the original matrix $\lambda$ and the combinatorics of $K$. 

\newpage
\setlength{\unitlength}{8.5mm}
\newcounter{qn}
\hoffset=0.09in
\setlength{\oddsidemargin}{12pt}
\begin{picture}(10,50)

\put(0.2,48.3){$\bm{v_{12}}$}
\put(1,48.3){$\cdots$}
\put(2,48.3){$\bm{v_{1j_1}}$}

\put(0,44){\framebox(3,4)[c]{\huge{$I_{j_{1}-1}$}}}
\put(0,40){\framebox(3,4)[c]{\huge{$0$}}}
\put(0,36){\framebox(3,4)[c]
{\begin{picture}(4,4)
\put(2,1){$\centerdot$}
\put(2,2){$\centerdot$}
\put(2,3){$\centerdot$}
\end{picture}}}
\put(0,32){\framebox(3,4)[c]{\huge{$0$}}}
\put(0,28){\framebox(3,4)[c]{\huge{$0$}}}

\put(3.2,48.3){$\bm{v_{22}}$}
\put(4,48.3){$\cdots$}
\put(5,48.3){$\bm{v_{2j_2}}$}

\put(3,44){\framebox(3,4)[c]{\huge{$0$}}}
\put(3,40){\framebox(3,4)[c]{\huge{$I_{j_{2}-1}$}}}
\put(3,36){\framebox(3,4)[c]{\huge{$0$}}}
\put(3,32){\framebox(3,4)[c]
{\begin{picture}(4,4)
\put(2,1){$\centerdot$}
\put(2,2){$\centerdot$}
\put(2,3){$\centerdot$}
\end{picture}}}
\put(3,28){\framebox(3,4)[c]{\huge{$0$}}}

\put(7.15,48.3){$\cdots$}

\put(6,44){\framebox(3,4)[c]{$\centerdot\;\centerdot\;\centerdot$}}
\put(6,40){\framebox(3,4)[c]{\huge{$0$}}}
\put(6,36){\framebox(3,4)[c]
{\begin{picture}(4,4)
\put(2.5,1){$\centerdot$}
\put(1.85,2){$\centerdot$}
\put(1.2,3){$\centerdot$}
\end{picture}}}
\put(6,32){\framebox(3,4)[c]{\huge{$0$}}}
\put(6,28){\framebox(3,4)[c]{\huge{$0$}}}

\put(9.2,48.3){$\bm{v_{m2}}$}
\put(10.1,48.3){$\cdots$}
\put(10.8,48.3){$\bm{v_{mj_m}}$}

\put(9,44){\framebox(3,4)[c]{\huge{$0$}}}
\put(9,40){\framebox(3,4)[c]{\huge{$0$}}}
\put(9,36){\framebox(3,4)[c]{\huge{$0$}}}
\put(9,32){\framebox(3,4)[c]{\huge{$I_{j_{m}-1}$}}}
\put(9,28){\framebox(3,4)[c]{\huge{$0$}}}

\put(12.3,48.3){$\bm{v_{11}}$}
\put(13.2,48.3){$\bm{v_{21}}$}
\put(14.2,48.3){$\cdots$}
\put(15.1,48.3){$\bm{v_{m1}}$}
\put(12,44){\framebox(4,4)[l]{\begin{picture}(4,4)
\put(0.1,3.5){$-1$}
\put(0.1,3){$-1$}
\put(0.1,1.3){\begin{picture}(1,1)
\put(0.36,1){$\centerdot$}
\put(0.36,0.5){$\centerdot$}
\put(0.36,0){$\centerdot$}
\end{picture}}
\put(0.1,0.2){$-1$}
\end{picture}}}
\put(14,45.65){{\huge{$0$}}}

\put(12,40){\framebox(4,4)[l]{\begin{picture}(4,4)
\put(0.2,3.5){$0$}
\put(0.2,3){$0$}
\put(0.2,1.3){\begin{picture}(1,1)
\put(0.2,1){$\centerdot$}
\put(0.2,0.5){$\centerdot$}
\put(0.2,0){$\centerdot$}
\end{picture}}
\put(0.2,0.2){$0$}
\put(0.8,3.5){$-1$}
\put(0.8,3){$-1$}
\put(0.5,1.3){\begin{picture}(1,1)
\put(0.7,1){$\centerdot$}
\put(0.7,0.5){$\centerdot$}
\put(0.7,0){$\centerdot$}
\end{picture}}
\put(0.8,0.2){$-1$}
\end{picture}}}
\put(14,41.65){{\huge{$0$}}}

\put(12,36){\framebox(4,4)[c]
{\begin{picture}(4,4)
\put(2,1){$\centerdot$}
\put(2,2){$\centerdot$}
\put(2,3){$\centerdot$}
\end{picture}}}

\put(12,32){\framebox(4,4)[l]{\begin{picture}(4,4)
\put(1.9,1.7){{\huge{$0$}}}
\put(3.2,3.5){$-1$}
\put(3.2,3){$-1$}
\put(3,1.3){\begin{picture}(1,1)
\put(0.57,1){$\centerdot$}
\put(0.57,0.5){$\centerdot$}
\put(0.57,0){$\centerdot$}
\end{picture}}
\put(3.2,0.2){$-1$}
\end{picture}}}

\put(12,28){\framebox(4,4)[c]{\huge{$\lambda$}}}
\put(12.3,27.5){$\bm{1}$}
\put(13,27.5){$\bm{2}$}
\put(14,27.5){$\bm{\cdots}$}
\put(15.4,27.5){$\bm{m}$}

\put(16.2,31.5){$\bm{1}$}
\put(16.2,30.9){$\bm{2}$}
\put(15.7,29.25){\begin{picture}(1,1)
\put(0.57,1){$\centerdot$}
\put(0.57,0.5){$\centerdot$}
\put(0.57,0){$\centerdot$}
\end{picture}}
\put(16.2,28.15){$\bm{n}$}

\put(4.6,25.8){\large{{\bf Figure 1.}} The matrix $\lambda(J)$}
\end{picture}

\setlength{\vfuzz}{2mm} 
\setlength{\textwidth}{160mm}
\setlength{\textheight}{205mm} 
\setlength{\oddsidemargin}{0pt}
\setlength{\evensidemargin}{0pt}
\hoffset=0.0in


Also in \cite{bbcg3} is an interpretation of this construction of the toric manifolds $M(J)$
in terms of generalized moment-angle complexes. To see this, consider the family of CW pairs
$$(\underline{D}^{2J}, \underline{S}^{2J-1}) \;=\; \big\{(D^{2j_{i}}, S^{2j_{i}})\big\}_{i=1}^{m}$$

\nd and the associated generalized moment-angle complex $Z(K_{P};(\underline{D}^{2J}, \underline{S}^{2J-1}))$.
There is an inclusion of tori $T^{m}\longrightarrow T^{d(J)}$ which includes the $i^{\rm{th}}$ circle in $T^{m}$ by the diagonal
\begin{equation}\label{eqn:diagonal}
S^{1} \longrightarrow (S^{1})^{j_{i}}.
\end{equation}

\nd This gives and action of $T^{m}$ on $Z\big(K_{P}(J); (D^2, S^1))\big)$.
Also, via a choice of diffeomorphism $D^{2j_{i}} \cong (D^{2})^{j_{I}}$, there is an
action of $T^{d(j)}$ on the moment-angle complex $Z(K_{P};(\underline{D}^{2J}, \underline{S}^{2J-1}))$.
With this understood, there are $T^{m}$ and $T^{d(J)}$-equivariant diffeomorphisms
\begin{equation}\label{eqn:exponent}
Z(K_{P};(\underline{D}^{2J}, \underline{S}^{2J-1}))\; \longrightarrow\; Z\big(K_{P}(J); (D^2, S^1)\big)
\end{equation}

\nd from which arises a diffeomorphism of orbit spaces
\begin{equation}\label{eqn:orbit.homeo}
Z(K_{P};(\underline{D}^{2J}, \underline{S}^{2J-1}))\big/\text{{\rm ker}}\;\!\lambda
\; \longrightarrow \; Z\big(K_{P}(J); (D^2, S^1))\big)\big/\text{{\rm ker}}\;\lambda(J)
\end{equation}

\nd which defines $M(J)$. (Here, $\text{{\rm ker}}\;\!\lambda$ and $\text{{\rm ker}}\;\lambda(J)$ are 
isomorphic subgroups of $T^{d(J)}$.) The appearance of the toric manifold $M(J)$ as the right hand side
of \eqref{eqn:orbit.homeo} is perplexing because that space is not reflected in either the fundamental construction
\eqref{eqn:defn.tm} or in diagram \eqref{eqn:djdiagram}. The matter is resolved by the diagram of 
Theorem \ref{thm:diagram} where the space appears in the bottom right  of the diagram with
$X_{i} = S^{2j_{i}-1}$, and $CX_{i} = D^{2j_{i}}$; the group $S^{1}$ acts freely on $S^{2j_{i}-1}$ in the usual way.
this observation is formalized in the next theorem.

\begin{thm}\label{thm:mj} The toric manifolds $M(J)$, defined by either 
the original Davis-Januszkiewicz construction \eqref{eqn:defn.tm} or equivalently, by the quotients 
\eqref{eqn:orbit.homeo}, are examples of Construction \ref{def:tilde2} as follows
$$M(J)\;=\;\big(T^{d(J)-m+n} \times P(J)\big)\big/\hspace{-2mm}\sim_{\lambda(J)} \;\cong \;
\big(S^{2j_{1}-1} \times S^{2j_{2}-1} \times \cdots \times S^{2j_{m}-1}\big/{\rm{ker}}\hspace{0.02in}\lambda\big) 
\times P^n \big/\!\!\sim_{2}.$$
\end{thm}
\begin{proof}
This result follows immediately from the right hand side of diagram \eqref{eqn:new.diagram}. \end{proof}

\nd Notice here that the advantage of the right hand side is the use of the (generalized) Davis-
Januszkiewicz construction with the polytope $P^{n}$, which is in general much smaller than $P(J)$ and has simpler
combinatorics.
\begin{remark}\label{rem:lastcol}
Notice that the part of $\lambda(J)$ in Figure $1$ reproduced below,
\skp{0.5}
\includegraphics[width=6in]{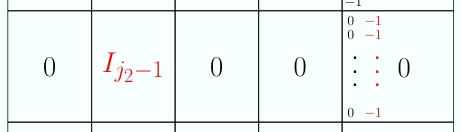}
\skp{0.1}
\centerline{{\bf Figure 2}}
\skp{0.4}
\nd is essentially the $(j_{2}-1)\times j_{2}$ matrix 

\begin{equation}\label{eqn:cplambda}
\left[\begin{array}{rrrrr}
\color{red}1\color{black}    &0     &\cdots&  0  &\color{red}-1\color{black}\\
0    &\color{red}1\color{black}     &\cdots&   0 &\color{red}-1\color{black}\\
\cdot&      &\cdots&\cdot    &\cdot\\
\cdot&      &\cdots&\cdot    &\cdot\\
0      &0    &\cdots&\color{red}1\color{black}   &\color{red}-1\color{black}
\end{array}\right]
\end{equation}

\

\nd which is the characteristic matrix for the diagonal $S^{1}$ action on
 $S^{2j_{2}-1}$. Indeed, the $i^{\rm{th}}$ ``block row'' of $\lambda(J)$ is the characteristic matrix for the diagonal $S^{1}$ action on
 $S^{2j_{i}-1}$. This particular connection to odd spheres becomes evident in the light of Construction \ref{def:tilde2} but
 was not obvious at the time that \cite{bbcg3} was written. This observation becomes relevant in the next section.
\end{remark}

\subsection{A simple illustration}
The toric manifold $\mathbb{C}P^{2}$ is made usually by the construction \eqref{eqn:defn.tm} using a two-simplex 
as the simple polytope.
The diagram below illustrates Construction \ref{def:tilde2} in this case. The ingredients are as follows:

\begin{enumerate}[(1)]\itemsep3pt
\item $P^{n} =  \Delta^{1}$ a one-simplex. Here $n=1$ and $m=2$,
\item $J = (1,2)$ so that $X_{1} = S^{1}$ and $X_{2} = S^{3}$ with the usual free $S^{1}$ action,
\item the characteristic map $\lambda\colon \mathbb{Z}^{2}\longrightarrow \mathbb{Z}$ is given by the matrix
$[1,-1]$ and $\text{ker}\hspace{0.02in}\lambda \cong T^{1}$ sits inside $T^{2}$ as $t \mapsto (t,t^{-1})$.
\end{enumerate}

\begin{rem}
Here, $P(J) = \Delta^{1}(1,2) = \Delta^{2}$ a two-simplex,  the usual polytope used to construct
$\mathbb{C}P^{2}$ as a toric manifold from \eqref{eqn:defn.tm}.
\end{rem}

\nd \hspace{0.0in}\includegraphics[width=6in]{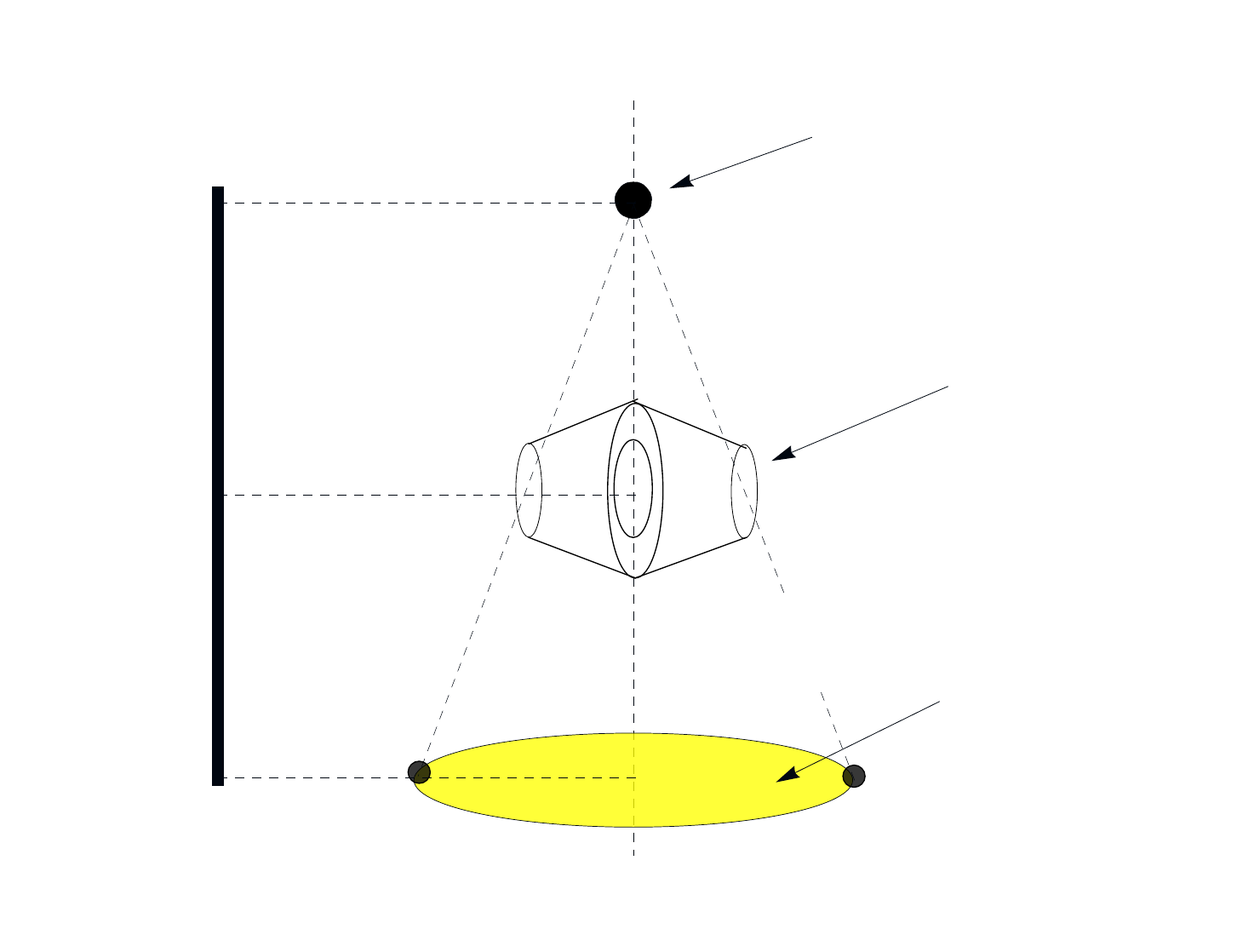}
\skp{-2}\hspace{0.65in}$P = \Delta^{1}$

\skp{-9.3}\hspace{3in}{$(S^{1}\times S^{3})\big/\text{ker}\hspace{0.02in}\lambda \rightsquigarrow 
S^{1}\big/\text{ker}\hspace{0.02in}\lambda = \ast$}

\skp{2.4}\hspace{4.1in}{$(S^{1}\times S^{3})\big/\text{ker}\hspace{0.02in}\lambda \cong S^{3}$}

\skp{3.1}\hspace{3.9in}{$(S^{1}\times S^{3})\big/\text{ker}\hspace{0.02in}\lambda 
\rightsquigarrow S^{3}\big/\text{ker}\hspace{0.02in}\lambda$
\skp{0.1}\hspace{5.1in}$=\; \mathbb{C}P^{1}$}

\nd \skp{2.2}\hspace{1.4in}{\large $\mathbb{C}P^{2} \;\cong\; 
\big(S^{1} \times S^{3}\big/{\rm{ker}}\hspace{0.02in}\lambda\big) 
\times \Delta^{1} \big/\!\!\sim_{2}$}

\

\nd In the diagram, the symbol $\rightsquigarrow$  represents the projection \eqref{eqn:quotients} 
which appears in part (ii) of Construction \ref{def:tilde2}.
The diagram presents  $\mathbb{C}P^{2}$ as the cone on $S^{3}$ attached to $\mathbb{C}P^{1}$
via the Hopf map

$$(S^{1}\times S^{3})\big/\text{ker}\hspace{0.02in}\lambda \stackrel{\pi_{2}}{\longrightarrow} 
S^{3}\big/\text{ker}\hspace{0.02in}\lambda.$$

\

\section{Application to iterated polyhedral products}\label{sec:iterated.pp}
The odd spheres of Section \ref{sec:families} are themselves examples of moment-angle complexes
$$S^{2j_{i}-1} \cong Z\big(K_{\Delta^{j_{i}-1}};\;(D^{2},S^{1})\big)$$

\nd where $K_{\Delta^{j_{i}-1}}$ is the simplicial complex dual to the boundary of the simplex $\Delta^{j_{i}-1}$.
Every moment-angle complex $Z\big(K;(D^{2},S^{1})\big)$ supports a free circle action and so it's natural to
ask about the case $X_{i} = Z\big(K_{i};(D^{2},S^{1})\big)$ in Construction \ref{def:tilde2} for a collection
$\{K_{1},K_{2},\ldots,K_{m}\}$ of arbitrary simplicial complexes. In this case, \eqref{eqn:alpha} becomes
\begin{equation}\label{eqn:czz}
(X_{1} \times X_{2} \times \cdots \times X_{m}) 
\times P^{n}\big/\!\!\sim_{1}\;\cong\;  Z\big(K_{P}; \big(\underline{CZ(K_{i};(D^{2},S^{1}))}, \;
\underline{Z(K_{i};(D^{2},S^{1}))}\big)\big).
\end{equation}

\

\subsection{A generalization of the construction $K(J)$}
\nd The problem of finding an analogue of \eqref{eqn:exponent} and \eqref{eqn:orbit.homeo} now presents
itself. Those diffeomorphisms  follow from \cite[Theorem 7.2]{bbcg3} which is a more general 
result about the behaviour of polyhedral products with respect to ``exponentiation'' of CW pairs. 
Recent work by Anton Ayzenberg \cite{aa}, generalizing this exponentiation construction becomes relevant
to understanding the problem further. A brief description of Ayzenberg's construction, tailored to the context
here, follows.

\

Let $K$ be a simplicial complex  on $m$ vertices and $\{K_{1}, K_{2}, \ldots, K_{m}\}$ a collection of
$m$ simplicial complexes on $j_{1},j_{2},\ldots, j_{m}$ vertices respectively.  From these ingredients, a new 
simplicial complex $K(K_{1}, K_{2}, \ldots, K_{m})$, on $j_{1}+j_{2}+\cdots+ j_{m}$ vertices, is constructed by
\begin{equation}\label{eqn:aa}
K(K_{1}, K_{2}, \ldots, K_{m})\; = \; \bigcup_{\sigma\in K}V_{\sigma}
\;\subset\; \Delta^{j_{1}-1}\ast \Delta^{j_{2}-1}\ast \cdots \ast \Delta^{j_{m}-1}
\end{equation}

\nd where
\begin{equation*}
V_{\sigma} = B_{1}\ast B_{2}\ast \cdots \ast B_{m}\quad {\rm with}\quad
B_i=\left\{\begin{array}{lcl}
\Delta^{j_{i}-1}  &{\rm if} & i\in \sigma\\
K_i &{\rm if} & i\notin \sigma.
\end{array}\right.
\end{equation*}

\begin{remark}
In this language, the construction $K(J)$ at the beginning of subsection \ref{subsec:oddspheres}  is just 
$K\big(\partial\Delta^{j_{1}-1}, \partial\Delta^{j_{2}-1},\ldots,\partial\Delta^{j_{m}-1}\big)$
where $\partial\Delta^{j_{i}-1}$ is the boundary of the $(j_{i}-1)$-simplex.
\end{remark}

\nd For $K = K_{P}$, the result analogous to \eqref{eqn:exponent} is the following.

\begin{thm}\cite[Proposition $5.1$]{aa}\label{thm:aa}
$$Z\big(K_{P}; \big(\underline{(D^{2})^{j_{i}}}, \;
\underline{Z(K_{i};(D^{2},S^{1})}\big)\big) \;=\; Z\big(K_{P}(K_{1},K_{2},\ldots,K_{m});(D^{2} ,S^{1})\big).$$
\end{thm}

\nd The next task is to relate these spaces to the right hand side of \eqref{eqn:czz}. The following proposition
addresses this point. As usual, set $d(J) = j_{1}+j_{2}+\cdots+ j_{m}$ and, to simplify the notation, set
$$Z(K_{i}) \;\becomes\; Z(K_{i};(D^{2},S^{1})).$$

\begin{prop}\label{prop:he}
There is  a $T^{d(J)}\;$-equivariant homotopy equivalence of polyhedral products
\begin{equation}\label{eqn:hediscs}
Z\big(K_{P}; \big(\underline{CZ(K_{i})}, \; \underline{Z(K_{i})}\big)\big) \simeq 
Z\big(K_{P}; \big(\underline{(D^{2})^{j_{i}}}, \; \underline{Z(K_{i})}\big)\big).
\end{equation}
\end{prop}

\begin{proof}
\nd Let $K$ be a simplicial complex on $j$ vertices. Then If $K$ is not the  $(j-1)$-simplex $\Delta^{j-1}$,  there is
simplicial embedding $K \longrightarrow  \partial{\Delta^{j-1}}$
\nd into the boundary. This induces an inclusion
\begin{equation}\label{eqn:embed}
Z\big(K;(D^{2},S^{1})\big) \longrightarrow Z\big(\partial{\Delta^{j-1}};(D^{2},S^{1})\big)
= \; \partial\big((D^{2})^{j}\big) \;\cong\;  S^{2j-1}
\end{equation}

\nd equivariant with respect to the action of the $j$-torus $T^{j}$. In turn, this
extends to an equivariant homotopy equivalence on cones:
$$CZ\big(K;(D^{2},S^{1})\big) \longrightarrow CS^{2j-1} \simeq D^{2j}$$ 

\nd where the action of $T^{j}$ preserves the  cone parameter. Next, choose a standard $T^{j}$-- equivariant diffeomorphism $h\colon D^{2j} \longrightarrow (D^{2})^{j}$ to get  an equivariant homotopy equivalence of CW pairs
\begin{equation}\label{eqn:pairs}
\big(CZ\big(K;(D^{2},S^{1})\big), Z\big(K;(D^{2},S^{1})\big)\big)\; \stackrel{h}{\longrightarrow}
\; \big((D^{2})^{j}, Z\big(K;(D^{2},S^{1})\big)\big).
\end{equation}

\nd The functorial properties of the polyhedral product \cite[Lemma 2.2.1]{denham.suciu} and an application of 
\eqref{eqn:pairs} for each $i = 1,2,\ldots,m$ completes the proof.\end{proof}

\begin{rem} In the case that $j = 4$ and $K$ is dual to the boundary of the square, the inclusion \eqref{eqn:embed} is
$$Z\big(K;(D^{2},S^{1})\big) \;=\; S^{3}\times S^{3}\; \longrightarrow \;
Z\big(\partial{\Delta^{3}};(D^{2},S^{1})\big)\; \simeq\; S^{7}\qquad \big(\simeq\; S^{3}\ast S^{3}\big)$$

\nd and the corresponding homotopy equivalence of pairs is
 $$\big(C(S^{3}\times S^{3}), \;S^{3}\times S^{3}\big)\;
\longrightarrow \; \big((D^{2})^{4},\; S^{3}\times S^{3}\big).$$

\nd equivariant with respect to the action of $T^{4}$.
\end{rem}

\subsection{The case of moment-angle complexes}
As before, let  $P^{n}$ be a simple polytope having $m$ facets, equipped with 
a {\em characteristic\/} function 
$$\lambda \colon \mathcal{F} \longrightarrow \mathbb{Z}^n$$

\nd satisfying the regularity condition following \eqref{eqn:lambda}. Regarding  
$$\text{ker}\hspace{0.02in}\lambda \hookrightarrow T^{m} \hookrightarrow T^{d(J)}$$ 

\nd as in \eqref{eqn:diagonal} and the case of odd spheres,  there is a natural {\em free\/} action of 
$\text{ker}\hspace{0.02in}\lambda$
on both sides of \eqref{eqn:hediscs} yielding a homotopy equivalence of orbit spaces
\begin{equation}\label{eqn:heorbits}
Z\big(K_{P}; \big(\underline{CZ(K_{i})}, \; \underline{Z(K_{i})}\big)\big)\big/\text{ker}\hspace{0.02in}\lambda\; \simeq \;
Z\big(K_{P}; \big(\underline{(D^{2})^{j_{i}}}, \; \underline{Z(K_{i})}\big)\big)\big/\text{ker}\hspace{0.02in}\lambda.
\end{equation}

\nd Combining Theorems \ref{eqn:new.diagram}, \ref{eqn:heorbits} and \eqref{eqn:heorbits} gives now
the main observation of this section.

\begin{thm}
For a simple polytope $P^{n}$, characteristic function $\lambda$ and $X_{i} = Z(K_{i})$, Construction \ref{def:tilde2} corresponds, up to homotopy, to a quotient of a moment-angle complex
by a free action of $\text{ker}\hspace{0.02in}\lambda$  as follows:
\begin{align*}
\big(X_1 \times X_2 \times \cdots \times X_m\big/{\rm{ker}}\hspace{0.02in}\lambda\big) 
\times P^n \big/\!\!\sim_{2} &\;\cong\;  Z\big(K_{P}; \big(\underline{CZ(K_{i})}, 
\underline{Z(K_{i})}\big)\big)\big/\rm{ker}\hspace{0.02in}\lambda \\
&\;\simeq\; Z\big(K_{P}; \big(\underline{(D^{2})^{j_{i}}}, 
\underline{Z(K_{i}}\big)\big)\big/\rm{ker}\hspace{0.02in}\lambda\\
&\; \cong\; Z\big(K_{P}(K_{1},K_{2},\ldots,K_{m}); (D^{2} ,S^{1})\big)\big/\rm{ker}\hspace{0.02in}\lambda. 
\end{align*}
\end{thm}

\nd It should be noted that in general, the space
$Z\big(K_{P}(K_{1},K_{2},\ldots,K_{m}); (D^{2} ,S^{1})\big)\big/\rm{ker}\hspace{0.02in}\lambda\;$ might not be a manifold
because $K_{P}(K_{1},K_{2},\ldots,K_{m})$ is dual to the boundary of a simple polytope only in the case that all the $K_{i}$ are boundaries of simplices.

\

\section{Further generalizations}\label{sec:further.pp}
Away from the diagonal circle action, the situation becomes a little more complicated. Ayzenberg's
construction \cite{aa}, can be done in the realm of polytopes. In particular, given a simple polytope
$P^{n}$ having $m$ facets and a sequence of simple polytopes
$\{P_{1},P_{2},\ldots,P_{m}\}$ where $P_{i} = P_{i}^{n_{i}}$ is a simple polytope of dimension $n_{i}$ having $j_{i}$ facets, the construction yields a new polytope
$$P^{n} \longrightarrow P^{n}\big(P_{1}, P_{2}, \ldots,P_{m}\big).$$

Though the new polytope is not simple when the $P_{i}$ differ from simplices, it does retain some
nice properties. A simplicial complex $K_{P(P_{1}, P_{2}, \ldots, P_{m})}$ is associated to it as the {\em nerve complex\/}. On the level of simplicial complexes, the construction is written
$$K_{P} \longrightarrow K_{P}(K_{P_{1}}, K_{P_{2}}, \ldots, K_{P_{m}})$$

\nd as in \eqref{eqn:aa}. As expected, it is shown in \cite{aa} that
$$K_{P(P_{1}, P_{2}, \ldots, P_{m})} \;=\; K_{P}(K_{P_{1}}, K_{P_{2}}, \ldots, K_{P_{m}}).$$

\nd Under this operation, the numbers $m$ and $n$ transform by the analogue of \eqref{eqn:transform1}:
\begin{equation}\label{eqn:transform2}
\left[\begin{array}{c}
m\\
n\\
m-n
\end{array}\right]
\rightsquigarrow
\left[\begin{array}{c}
d(J) = j_{1}+j_{2}+\cdots+j_{m}\\
n + N\\
d(J) - n - N
\end{array}\right]
\end{equation}

\nd where $N =  n_{1}+n_{2}+\cdots + n_{m}$. 
Notice that \eqref{eqn:transform2} reduces to \eqref{eqn:transform1} for the case $P_{i} = \Delta^{j_{i}-1}$.

\subsection{A full torus action} For each $i$ = 1,2,\ldots,m, let 
$$\lambda_{i}  \colon \mathbb{Z}^{j_{i}} \longrightarrow \mathbb{Z}^{n_{i}}$$

\nd be a characteristic function on $P_{i}$. It has 
$\text{ker}\hspace{0.02in}\lambda_{i} \cong T^{j_{i}-n_{i}}\;$ which acts freely on 
$Z\big(K_{P_{i}}; (D^{2},S^{1})\big)$.
The next step is to mimic the construction of $\lambda(J)$ in section \ref{subsec:oddspheres}.
To this end, denote by $\overline{\lambda}_{i}$ the first $j_{i}-1$ columns of the 
$(n_{i}\times j_{i})$--matrix $(\lambda_{i}^{lk})$ corresponding to $\lambda_{i}$. The last column of 
$(\lambda_{i}^{ lk})$ is 

\begin{equation}\label{eqn:lastcol}
\begin{bmatrix}
\lambda_{i}^{ 1j_{i}}\\
\\
\lambda_{i}^{ 2j_{i}}\\
\cdot\\
\cdot\\
\cdot\\
\lambda_{i}^{ n_{i}j_{i}}
\end{bmatrix}
\end{equation}
\skp{0.25}
\nd Entirely by analogy with the case of odd spheres in section \ref{subsec:oddspheres}, in particular
Remark \ref{rem:lastcol}, the matrices
$\lambda,\; \lambda_{1}, \lambda_{2},\ldots,\lambda_{m}$ are used to construct
an $\big((n + N)\times d(J)\big)$--matrix $\lambda(J,N)$, shown in Figure 3,
which defines a map
$$\lambda(J,N)\colon\; \mathbb{Z}^{d(J)} \longrightarrow \mathbb{Z}^{n+N}.$$

\nd The characteristic matrix corresponding to the diagonal $S^{1}$ action on the odd sphere
$S^{2j_{i}-1}$ which appears in \eqref{eqn:cplambda}, is replaced by $\lambda_{i}$  which has has 
$\text{ker}\hspace{0.02in}\lambda_{i} \cong T^{j_{i}-n_{i}}\;$ acting freely on 
$Z\big(K_{P_{i}}; (D^{2},S^{1})\big)$. The blocks $I_{j_{i}-1}$ in Figure 1 are replaced by 
$\overline{\lambda}_{i}$ and the last columns of ``$-1$'', by \eqref{eqn:lastcol}.

\newpage
\setlength{\unitlength}{8.5mm}
\hoffset=0.0in
\voffset=0.0in
\setlength{\oddsidemargin}{12pt}
\begin{picture}(10,50)


\put(0,44){\framebox(3,4)[c]{\huge{$\overline{\lambda}_{1}$}}}
\put(0,40){\framebox(3,4)[c]{\huge{$0$}}}
\put(0,36){\framebox(3,4)[c]
{\begin{picture}(4,4)
\put(2,1){$\centerdot$}
\put(2,2){$\centerdot$}
\put(2,3){$\centerdot$}
\end{picture}}}
\put(0,32){\framebox(3,4)[c]{\huge{$0$}}}
\put(0,28){\framebox(3,4)[c]{\huge{$0$}}}


\put(3,44){\framebox(3,4)[c]{\huge{$0$}}}
\put(3,40){\framebox(3,4)[c]{\huge{$\overline{\lambda}_{2}$}}}
\put(3,36){\framebox(3,4)[c]{\huge{$0$}}}
\put(3,32){\framebox(3,4)[c]
{\begin{picture}(4,4)
\put(2,1){$\centerdot$}
\put(2,2){$\centerdot$}
\put(2,3){$\centerdot$}
\end{picture}}}
\put(3,28){\framebox(3,4)[c]{\huge{$0$}}}


\put(6,44){\framebox(3,4)[c]{$\centerdot\;\centerdot\;\centerdot$}}
\put(6,40){\framebox(3,4)[c]{\huge{$0$}}}
\put(6,36){\framebox(3,4)[c]
{\begin{picture}(4,4)
\put(2.5,1){$\centerdot$}
\put(1.85,2){$\centerdot$}
\put(1.2,3){$\centerdot$}
\end{picture}}}
\put(6,32){\framebox(3,4)[c]{\huge{$0$}}}
\put(6,28){\framebox(3,4)[c]{\huge{$0$}}}


\put(9,44){\framebox(3,4)[c]{\huge{$0$}}}
\put(9,40){\framebox(3,4)[c]{\huge{$0$}}}
\put(9,36){\framebox(3,4)[c]{\huge{$0$}}}
\put(9,32){\framebox(3,4)[c]{\huge{$\overline{\lambda}_{m}$}}}
\put(9,28){\framebox(3,4)[c]{\huge{$0$}}}

\put(12,44){\framebox(4,4)[l]{\begin{picture}(4,4)
\put(0.1,3.5){$\lambda_{1}^{1j_{1}}$}
\put(0.1,2.7){$\lambda_{1}^{2j_{1}}$}
\put(0.1,1.3){\begin{picture}(1,1)
\put(0.36,1){$\centerdot$}
\put(0.36,0.5){$\centerdot$}
\put(0.36,0){$\centerdot$}
\end{picture}}
\put(0.1,0.2){$\lambda_{1}^{n_{1}j_{1}}$}
\end{picture}}}
\put(14,45.65){{\huge{$0$}}}

\put(12,40){\framebox(4,4)[l]{\begin{picture}(4,4)
\put(0.2,3.5){$0$}
\put(0.2,2.7){$0$}
\put(0.2,1.3){\begin{picture}(1,1)
\put(0.2,1){$\centerdot$}
\put(0.2,0.5){$\centerdot$}
\put(0.2,0){$\centerdot$}
\end{picture}}
\put(0.2,0.2){$0$}
\put(0.8,3.5){$\lambda_{2}^{1j_{2}}$}
\put(0.8,2.7){$\lambda_{2}^{2j_{2}}$}
\put(0.5,1.3){\begin{picture}(1,1)
\put(0.5,1){$\centerdot$}
\put(0.5,0.5){$\centerdot$}
\put(0.5,0){$\centerdot$}
\end{picture}}
\put(0.8,0.2){$\lambda_{2}^{n_{2}j_{2}}$}
\end{picture}}}
\put(14,41.65){{\huge{$0$}}}

\put(12,36){\framebox(4,4)[c]
{\begin{picture}(4,4)
\put(2,1){$\centerdot$}
\put(2,2){$\centerdot$}
\put(2,3){$\centerdot$}
\end{picture}}}

\put(12,32){\framebox(4,4)[l]{\begin{picture}(4,4)
\put(1.9,1.7){{\huge{$0$}}}
\put(2.8,3.5){$\lambda_{m}^{1j_{m}}$}
\put(2.8,2.7){$\lambda_{m}^{2j_{m}}$}
\put(2.8,1.3){\begin{picture}(1,1)
\put(0.4,1){$\centerdot$}
\put(0.4,0.5){$\centerdot$}
\put(0.4,0){$\centerdot$}
\end{picture}}
\put(2.8,0.2){$\lambda_{m}^{n_{m}j_{m}}$}
\end{picture}}}

\put(12,28){\framebox(4,4)[c]{\huge{$\lambda$}}}
\put(12.3,27.5){$\bm{1}$}
\put(13,27.5){$\bm{2}$}
\put(14,27.5){$\bm{\cdots}$}
\put(15.4,27.5){$\bm{m}$}

\put(16.2,31.5){$\bm{1}$}
\put(16.2,30.9){$\bm{2}$}
\put(15.7,29.25){\begin{picture}(1,1)
\put(0.57,1){$\centerdot$}
\put(0.57,0.5){$\centerdot$}
\put(0.57,0){$\centerdot$}
\end{picture}}
\put(16.2,28.15){$\bm{n}$}
\put(4.6,25.8){\large{{\bf Figure 3.}} The matrix $\lambda(N,J)$}
\end{picture}

\setlength{\vfuzz}{2mm} 
\setlength{\textwidth}{160mm}
\setlength{\textheight}{205mm} 
\setlength{\oddsidemargin}{0pt}
\setlength{\evensidemargin}{0pt}
\hoffset=0.0in

\subsection{The rank of the matrix $\lambda(J,N)$} The matrix corresponding to
$\lambda$ is a characteristic matrix and so  can be written in {\em refined\/} block form as:
$$\lambda = I_{n}\big|S$$

\nd  where $I_{n}$ is the $n\times n$-identity matrix and $S$ is of size\; $n\times (m-n)$.
Similarly, the matrix corresponding to $\overline{\lambda}_{i}$ can be written in the block form as 
$$I_{n_{i}}\big|S_{i}$$

\nd  where $I_{n_{i}}$ is the $n_{i}\times n_{i}$-identity matrix and $S_{i}$ is of size\;
$n_{i}\times (j_{i}-1-n_{i})$. This observation allows the conclusion that 
the row rank of $\lambda(J,N)$ is $N+n$ and the next proposition follows.

\begin{prop}
The row rank of the matrix $\lambda(J,N)$ is $N+n$ and so 
$${\rm{ker}}\hspace{0.02in}\lambda(J,N) \;\cong\; T^{d(J)-N-n}.$$
\end{prop}

\
 
\subsection{A new toric space construction} The inclusion
$\text{ker}\hspace{0.02in}\lambda(J,N) \longrightarrow T^{d(J)}$ gives an action of 
$\text{ker}\hspace{0.02in}\lambda(J,N)\;$ on the the $T^{d(J)}$-equivariantly
homotopy equivalent spaces 
$$Z\big(K_{P}; \big(\underline{CZ(K_{P_{i}})}, \;
\underline{Z(K_{P_{i}})}\big)\big)
\; \simeq\; Z\big(K_{P}; \big(\underline{(D^{2})^{j_{i}}}, \; \underline{Z(K_{P_{i}})}\big)\big).$$

\nd In this context of simple polytopes, Theorem \ref{thm:aa} gives
$$Z\big(K_{P}; \big(\underline{(D^{2})^{j_{i}}}, \; \underline{Z(K_{P_{i}})}\big)\big)\\
\; = \;Z\big(K_{P(P_{1}, P_{2}, \ldots, P_{m})};(D^{2} ,S^{1})\big).$$

\

\nd The next theorem follows from the assembly of this information.

\begin{thm}\label{thm:bigaction}
There is a homotopy equivalence of orbit spaces
$$Z\big(K_{P}; \big(\underline{CZ(K_{P_{i}})}, \;
\underline{Z(K_{P_{i}})}\big)\big)\big/{\rm{ker}}\hspace{0.02in}\lambda(J,N)
\;\simeq\; Z\big(K_{P(P_{1}, P_{2}, \ldots, P_{m})};(D^{2} ,S^{1})\big)\big/
{\rm{ker}}\hspace{0.02in}\lambda(J,N)$$
\end{thm}

\nd where $K_{P(P_{1}, P_{2})}$ is the nerve complex of the $d(J)$--faceted, $(n+N)$--dimensional polytope
$P(P_{1}, P_{2}, \ldots, P_{m})$ and $\text{ker}\hspace{0.02in}\lambda(J,N)$ is isomorphic to a torus of 
dimension $d(J) - (n+N)$.

\begin{rem}
 In a natural way, the form of the matrix $\lambda(J,N)$ indicates that
$$Q = \text{ker}\lambda_{1}\oplus \text{ker}\lambda_{2}\oplus \cdots \oplus \text{ker}\lambda_{m}$$

\nd can be considered as a $(d(J)-N)$-dimensional subspace of $\mathbb{Z}^{d(J)}$. Work in progress by the
authors has, among other things, the goal of characterizing the cases when ${\rm{ker}}\hspace{0.02in}\lambda(J,N)$
is an $(d(J)-N-n)$-dimensional subspace of $Q$. In particular, this would imply that the action in
Theorem \ref{thm:bigaction} is free. There would follow a natural generalization of Construction \ref{def:tilde2}
in which the free $S^{1}$ action of $X_{i}$ is replaced with a free $T^{j_{i}-n_{i}}$ action on each $X_{i}$.
In the case above, $X_{i} =  Z(K_{P_{1}})$ with $\text{ker}\hspace{0.02in}\lambda_{i} \cong T^{j_{i}-n_{i}}\;$
acting freely. The group $\text{ker}\hspace{0.02in}\lambda(J,N)$ acts on 
$Z(K_{P_{1}}) \times Z(K_{P_{2}}) \times \cdots \times Z(K_{P_{m}})$ via the inclusion 
$\text{ker}\hspace{0.02in}\lambda(J,N) \longrightarrow T^{d(J)}$. So, with very little change in the
definitions, the left hand side of the equivalence in Theorem \ref{thm:bigaction} would be identified as
$$\big(Z(K_{P_{1}}) \times Z(K_{P_{2}}) \times \cdots \times Z(K_{P_{m}})
\big/\text{ker}\lambda(J,N)\big) \times P^{n} \big/\!\!\sim_{2}.$$
\end{rem}

\skp{0.5}
\bibliographystyle{amsalpha}

\begin{thebibliography}{99}
\bibitem{aa} A.~Ayzenberg, {\em Composition of simplicial complexes, polytopes and 
multigraded Betti numbers\/}. Available at: http://arxiv.org/abs/1301.4459

\bibitem{bbcgpnas}A.~Bahri, M.~Bendersky, F.~R.~Cohen, and S.~Gitler, 
{\em Decompositions of the polyhedral product functor with applications to 
moment-angle complexes and related spaces}, PNAS 2009 {\bf 106}:12241-12244.

\bibitem{bbcg} A.~Bahri, M.~Bendersky, F.~Cohen and S.~Gitler, {\em The Polyhedral Product
Functor: a method  of computation for moment-angle complexes, arrangements and 
related spaces\/}. Advances in Mathematics, {\bf 225} (2010), 1634--1668.


\bibitem{bbcg3}  A.~Bahri, M.~Bendersky, F.~Cohen and S.~Gitler, {\em Operations on polyhedral 
products and a new topological construction of infinite families of toric manifolds}. 
Online at: http://arxiv.org/abs/1011.0094



\bibitem{buchstaber.panov.2} V.~Buchstaber and T.~Panov, {\em  Torus actions and their 
applications in topology and combinatorics\/}, AMS University Lecture Series, {\bf 24}, (2002).



\bibitem{davis.jan}  M.~Davis, and T.~Januszkiewicz, {\em Convex polytopes,
Coxeter orbifolds and torus actions}, Duke Math. J., \textbf{62}
(1991), no. 2, 417-–451.

\bibitem{denham.suciu} G.~Denham and A.~Suciu, {\em Moment-angle complexes,
monomial ideals and Massey products}, Pure and Applied Mathematics Quarterly
{\bf 3}, no. 1, (2007), 25--60.









\bibitem{pb} J.~S.~Provan and  L.~J.~Billera, {\em Decompositions of simplicial complexes
related to diameters of convex polyhedra}, Mathematics of Operations Research, volume 5,
(1980), 576--594.




\end{thebibliography}

\end{document}